\documentclass[11pt]{amsart}
\usepackage{amssymb,amsmath,amsthm,amsfonts,verbatim,tikz}
\usepackage[alphabetic]{amsrefs}
\usepackage[all,cmtip]{xy}
\usepackage{hyperref}
\usepackage[margin=1.25in]{geometry}
\usepackage{todonotes}
\usepackage{qtree}

\newtheorem{theorem}{Theorem}[section]

\newtheorem{corollary}[theorem]{Corollary}
\newtheorem{lemma}[theorem]{Lemma}

\newtheorem{definition}[theorem]{Definition}
\newtheorem{problem}[theorem]{Problem}

\theoremstyle{definition}

\newtheorem{example}[theorem]{Example}

\newtheorem{remark}[theorem]{Remark}
\newtheorem{remarks}[theorem]{Remarks}

\numberwithin{equation}{section}

\newcommand{\into}{\hookrightarrow}

\newcommand{\xdownarrow}[1]{%
  {\left\downarrow\vbox to #1{}\right.\kern-\nulldelimiterspace}
}

\newcommand{\Cb}{\mathbb{C}}

\newcommand{\Fb}{\mathbb{F}}
\newcommand{\F}{\mathbb{F}}

\newcommand{\Pb}{\mathbb{P}}

\newcommand{\Zb}{\mathbb{Z}}

\DeclareMathOperator{\PSL}{PSL}

\newcommand{\tC}{\widetilde{C}}

\DeclareMathOperator{\ed}{ed}

\DeclareMathOperator{\PGL}{PGL}

\DeclareMathOperator{\Gal}{Gal}

\DeclareMathOperator{\characteristic}{char}

\newcommand{\para}[1]{\medskip\noindent\textbf{#1.}}

\title{Essential dimension relative to branched covers of degree at most n}

\author{Benson Farb and Jesse Wolfson}
\address{Department of Mathematics, University of Chicago}
\email{farb@math.uchicago.edu}
\address{ \vskip -.5 cm Department of Mathematics, University of California-Irvine}
\email{wolfson@uci.edu}
\thanks{The authors are partially supported by NSF grants DMS- 2203355(BF), and
DMS-1944862 and DMS-2506184 (JW)}

\begin{document}
\maketitle
\begin{abstract}
We prove for various finite groups $G$ and integers $n\geq 1$ that there are families of equations with Galois group $G$ that cannot be simplified to a one-parameter family even after adjoining a root of a polynomial of degree at most $n$.  In more geometric language, there are $G$-varieties $X$ with the following property: for any $G$-equivariant branched cover $\widetilde{X}\to X$ of degree $\leq n$, there is no dominant rational $G$-map $\widetilde{X}\dashrightarrow C$ to any $G$-curve $C$.  The method  of proof is new, and applies in cases where previous methods do not.
\end{abstract}

\section{Introduction}
Let $k$ be a perfect field. A {\em $G$-variety} over $k$ is a $k$-variety $X$ equipped with a faithful action of a finite group $G$ on $X$ by birational automorphisms. 
A {\em $G$-compression} is a dominant rational map 
\[f:X\dashrightarrow Y\] of $G$-varieties; equivalently, the $G$-action on $X$ is the pullback via $f$ of the $G$-action on $Y$.  In classical language, a $G$-compression is a simplification of equations via a rational change of variables.  

\begin{example}[{\bf Kummer's theorem}]
Suppose that $\characteristic(k)\nmid n$ and that $k$ contains a primitive $n^{th}$ root of unity $\zeta\in k$. Then every $\Zb/n\Zb$-variety compresses to $\Pb^1$ with its standard $\Zb/n\Zb$-action $z\mapsto \zeta\cdot z$. In Galois-theoretic terms, every cyclic extension of a function field of $k$-varieties is given by ``adjoining an $n^{th}$ root''.
\end{example}

In contrast, Felix Klein proved that if $\characteristic(k)=0$ then there is an $A_5$-action on $\Pb^2$ that cannot be compressed to any $A_5$-action on a $1$-dimensional variety.  To state this result in modern terms we need the following definition of Buhler-Reichstein \cite{BR}.

\begin{definition}[{\bf Essential dimension}]
Let $k$ be a field. The {\em essential dimension} over $k$  of a  faithful $G$-variety $X$, denoted 
$\ed_k(X\dashrightarrow X/G)$ or $\ed_k(X)$, is the smallest $d\geq 1$ so that there is a $G$-compression $X\dashrightarrow Y$ over $k$ to a $d$-dimensional faithful $G$-variety $Y$.  
\end{definition}

Kummer's theorem gives $\ed_k(X)=1$ for every $\Zb/n\Zb$-variety $X$ over $k$ with a primitive $n^{th}$ root of unity; Klein's theorem gives $\ed_k(\Pb^2\to \Pb^2/A_5)=2$.\footnote{This holds for any field $k$ not containing $\mathbb{F}_4$, cf.  \cite[Proposition 5]{Ledet} and \cite[Theorem 1.6]{CHKZ}.}  In contrast to his incompressibility result for $A_5$, Klein proved that every $A_5$ extension of function fields is indeed icosahedral after adjoining a square root.  In more geometric language: 

\begin{theorem}[{\bf Klein's Normalformsatz}]
\label{theorem:normalformsatz}
Let $k$ be a field of characteristic 0 with $\sqrt{5}\in k$. Let $X$ be any $A_5$-variety over $k$.  Then $X$ has an 
$A_5$-equivariant branched cover\footnote{By a {\em degree n branched cover} we mean a generically $n$-to-$1$, dominant rational map $\tilde{X}\dashrightarrow X$.} $\widetilde X\dashrightarrow X$ of 
degree at most $2$  
 such that there is an 
$A_5$-compression \[\widetilde{X}\dashrightarrow \Pb^1.\] 
\end{theorem}

Klein's Normalformsatz is an example of a general classical problem, studied by Hamilton \cite{Ham}, Sylvester-Hammond \cite[p.1]{SH} and many others, which asks:  can one reduce the number of variables 
in a system of polynomials by adjoining the solutions of a lower degree polynomial?  In more geometric language \footnote{We leave it to the reader to write down the equivalent Galois-theoretic formulation.}:

\begin{problem}[{\bf Hamilton \cite{Ham}}]
\label{problem:main}
Let $k$ be a field. Compute, for a given faithful $G$-variety $X$ and $n\geq 1$,

\begin{equation}
\label{eq:reled}
\ed_k(X;\leq n):=\min\{\dim(Y) : \exists \ \widetilde{X}\stackrel{\leq n}{\dashrightarrow}X \ \text{and}\ \exists \ \text{$G$-compression}\  
\widetilde{X}\dashrightarrow Y\}
\end{equation}

\vspace{.2cm}
where the min ranges over all faithful $G$-varieties $\widetilde{X}$ and $Y$ over $k$ and all branched covers $\widetilde{X}\dashrightarrow X$ of degree at most $n$.  Further, for a given finite group $G$, compute 
\begin{equation}
\label{eq:reled2}
			\ed_k(G;\leq n):=\sup \ed_k(X;\leq n)
\end{equation}

\vspace{.2cm}
where the supremum is taken over all faithful $G$-varieties $X$.  \footnote{It is known that $\ed_k(G;\leq n)=\ed_k(V;\leq n)$ for any faithful linear $G$-variety $V$ (e.g. \cite[Example 4.6 and Lemma 4.9]{FKW2}). }
\end{problem}

\begin{remark}\mbox{}
	\begin{enumerate}
		\item The assumption that the $G$-actions are faithful is critical.  Without this, there is always the trivial map to a point (with constant $G$-action). 
		\item In classical language, the $G$-variety $X$ encodes the problem of solving for $x\in X$ such that $f(x)=y$ for given $y$, where $f\colon X\to X/G=Y$ is the quotient.  In this language, 		Problem~\ref{problem:main}  asks how simply an equation with Galois group $G$ can be solved using elimination theory and an accessory algebraic function of degree at most $n$. 
	\end{enumerate}
\end{remark}

One reason for Klein's and others' interest in Problem \ref{problem:main} 
is that many of the known solutions to classical equations, for example those involving modular functions and those in enumerative geometry,  are of this form; namely, where one can reduce the number of variables by adjoining the roots of a polynomial of lower degree.  See \cite{FW} and \cite{FKW2} for many examples.    Klein's theorems mentioned above can be written as:

\[\ed_k(A_5)=2\ \ \ \text{but}\ \ \ \ed_k(A_5;\le 2)=1\]
for any $k$ with $\characteristic(k)=0$ and with $\sqrt{5}\in k$. 

\vspace{.2cm}
While the literature of the last 200 years contains upper bounds for 
$\ed_{\Cb}(G;\leq n), n\geq 2$ for many examples, lower bounds are lacking, even in the simplest cases.   For example, Klein proved that any $\PSL_2(\Fb_7)$-variety $X$ has an at-most $4$-sheeted branched cover $\widetilde{X}$ that compresses to the Klein quartic curve, so that
\[ \ed_{\Cb}(\PSL_2(\F_7);\leq 4)=1.\]
Can one do better, replacing $n=4$ by $n=2$ or $n=3$? Corollary \ref{cor:main5} below implies that the answer is ``no'' for $n=2$; that is, $\ed_{\Cb}(\PSL_2(\F_7);\leq 2)>1$.  The case $n=3$ remains open.

\para{Results}  
The main technical result of this paper is the following. Its proof exploits the classical geometry of $G$-curves (see below).

\begin{theorem}[{\bf Main Theorem}]\label{t:main}
	Let $k$ be a perfect field. Let $n\ge 2$. Let $G$ be a finite group such that:
    \begin{enumerate}
        \item $G$ has no proper subgroup of index at most $n$ (in particular $|G|>n$), 
        \item $G$ contains a subgroup $M$ with $|M|>n$ that acts faithfully on $\Pb^1$ over $k$, and
        \item\label{a:cast} $G$ does not act nontrivially on a smooth curve of genus $g\le (n-1)^2$.
    \end{enumerate}
    Then $\ed_k(G;\le n)>1$.
\end{theorem}

Theorem \ref{t:main} is applicable because its three hypotheses are easy to check in examples.  Over $\Cb$, we can apply it to give the following.

\begin{corollary}[{\bf Sample results}]
\label{cor:main5}\mbox{}
\begin{enumerate}
\item\label{i:mKl}  Let $G$ be any non-abelian simple finite group except $A_5$. Then 
\[\ed_{\Cb}(G;\le 2)>1.\]
\item\label{i:mA7} $\ed_{\Cb}(A_7;\leq 6)>1$. 
\item\label{i:mPSL2} Let $p\geq 7$ be prime, and let $n\le\min\{p-1,1+\lfloor \sqrt{1+\frac{p(p^2-1)}{168}}\rfloor\}$ (note that for $p>163$ this min equals $p-1$).  Then \[\ed_{\Cb}(\PSL_2(\F_p);\le n)>1.\]
\end{enumerate}
\end{corollary}

\begin{remarks}\mbox{}
\begin{enumerate}
    \item Item~\ref{i:mKl} shows that Klein's {\it Normalformsatz} (Theorem \ref{theorem:normalformsatz}) is exceptional among finite simple groups. 
    \item In the spirit of Hilbert's 13th problem, Item~\ref{i:mA7} shows that the general degree 7 polynomial cannot be reduced to a 1-variable algebraic function even after allowing an accessory sextic. 
    \item Item~\ref{i:mPSL2} shows that for each $n\geq 1$ the theory of $\ed_{\Cb}(-;\le n)$ is nontrivial. 
    \item A finite group $G$ is the Galois group of a family of polynomials of degree $\mu(G)$, where $\mu(G)$ denotes the order of the smallest permutation representation. A sharpened version of Problem~\ref{problem:main} asks for lower bounds on $\ed_k(G,\le \mu(G)-1)$.\footnote{e.g. for $G=A_n$, this sharpened version asks how much we can simplify the general degree $n$ polynomial using only the solution of a single polynomial of lower degree.}  Item~\ref{i:mA7} of 
    Corollary~\ref{cor:main5} addresses this sharp version of Problem~\ref{problem:main} over $\Cb$ for $A_7$.  Similarly, for $p>11$, Galois showed that $\mu(\PSL_2(\F_p))=p+1$.  Thus, for $p>163$, the 
    $n$ in Item~\ref{i:mPSL2} is only off by 1 from the natural choice of $n=p$.
\end{enumerate}
\end{remarks}

\begin{remark}[{\bf Previous methods}]\label{r:prev}
All work up to this point has given lower bounds only for the version of Problem \ref{problem:main} where for a given prime $p$, {\em any} degree prime to $p$ branched cover $\widetilde{X}\to X$ is allowed; this is called  the {\em ``essential dimension at $p$''} and is denoted by $\ed(X;p)$.  See, e.g.\ \cite{BR2, ReiYo,KM,Rei,FKW,BF,FKW2,FKW3}.  These methods applied to $\ed(G;\leq n)$ give exactly the following:
\begin{equation}\label{e:comp}
    \ed_k(G;\le n)\ge  \max\{\{\ed_k(G;p)\}_{p>n},\{\ed_k(P;\le n)\}_{P\subset G~p\text{-Sylow},~p\le n}\}.
\end{equation}
Classical questions about the complexity of solving polynomials, e.g. Problem~\ref{problem:main}, impose a different set of requirements on the collection of branched covers allowed.  To tackle these it is necessary 
to move beyond what essential dimension at $p$ can give. 

As an example, we claim that for $k=\Cb$, $G=A_7$ and $n=6$ the inequality \eqref{e:comp} is strict, and so does not suffice to prove Corollary~\ref{cor:main5} ~\eqref{i:mA7}, since the right-hand side of  \eqref{e:comp} equals $1$ in this case.  To prove this claim, first 
note that 
\[
    |A_7|=3\cdot 4\cdot 5\cdot 6\cdot 7
\]
implies that $\ed_{\Cb}(A_7;p)\le 1$ for $p\ge 5$.  For $p=2,3$, the $p$-Sylow of $A_7$ is isomorphic to $\Zb/p\Zb\times\Zb/p\Zb$, and since $6=2\cdot 3$, we can kill off a $\Zb/p\Zb$ factor in both cases by adjoining a $6^{th}$ root. Kummer's theorem implies that the right-hand side of~\eqref{e:comp} for $A_7$ and $n=6$ equals 1, as claimed. Similar arguments show that the special cases of Corollary \ref{cor:main5}
\[\ed_{\Cb}(\PSL_2(\F_7);\leq 2)>1, \ \ \ed_{\Cb}(\PSL_2(\F_{11});\leq 3)>1, \ \ \text{and}\ \ \ed_{\Cb}(\PSL_2(\F_{13};\le 4)>1\]
give further examples where the inequality~\eqref{e:comp} is strict.  
\end{remark}
  
The outline of the proof of Theorem \ref{t:main} proceeds as follows.   We assume the theorem is false, and from this we construct a single $G$ curve $C$ to which every $G$-variety compresses after taking a degree $\leq n$ cover.  We then construct a $G$-curve to violate this.  The key invariant we use to prove certain curves cannot compress to others is the gonality of a curve.
\smallskip
\newline
\noindent
{\bf Acknowledgements.}
It is a pleasure to thank Mark Kisin for many comments, questions and discussions which helped sharpen and improve this paper.  We thank Curt McMullen and Zinovy Reichstein for helpful comments on a draft.

\section{Rational functions on curves}
We work throughout over a perfect field $k$. Our main tool is Castelnuovo's inequality (see \cite[Theorem 3.11.3]{Stichtenoth} and also \cite[Proposition 1]{Accola} for $k=\Cb$).

\begin{theorem}[Castelnuovo's Inequality]
	Let $C$ be an irreducible algebraic curve over a perfect field $k$. Let $f_i\colon C\to D_i$ be rational maps of curves of degree $n_i\geq 1$ for $i=1,2$.  Assume that the map $(f_1,f_2)\colon C\to D_1\times D_2$ is birational onto its image. 	Then
	\[
		g(C)\le n_1g(D_1)+n_2g(D_2)+(n_1-1)(n_2-1).
	\]
\end{theorem}

We need a slightly more general form of \cite[Corollary 3.11.4]{Stichtenoth}; presumably the following lemma was known to Riemann.

\begin{lemma}\label{l:castn}
	Let $C$ be an algebraic curve of genus $g(C)$ over a perfect field $k$. Let $j\geq 1$ and let $f_i\colon C\to \Pb^1$ (so $f_i\in k(C)$) have degree $n\geq 1$ for $i=1,\ldots,j$.   If $k(f_1,\ldots,f_j)=k(C)$, then 
	\[
		g(C)\le (n-1)^2.
	\]
\end{lemma}
\noindent 
We remark that $k(f_1,\ldots,f_j)$ is the function field of the curve that is the image of the map $C\to \Pb^1_1\times\cdots\times\Pb^1_j$ under the map $z\mapsto (f_1(z),\ldots ,f_j(z))$.

\begin{proof}
	We prove this by induction on $j$.  The case $j=2$ is exactly the ``Riemann Inequality'', stated as \cite[Corollary 3.11.4]{Stichtenoth}. For the induction step, let $F'=k(f_1,\ldots,f_{j-1})\subset k(C)$.  Denote by $g(F')$ the genus of the smooth projective curve with function field $F'$. Let $n/m=[k(C):F']$.  Then $[F':k(f_i)]=m$ for all $i=1,\ldots,j-1$, and by the inductive hypothesis, 
	\[
		g(F')\le (m-1)^2.
	\]
	By Castelnuovo's Inequality \cite[Theorem 3.11.3]{Stichtenoth}, 
	\begin{align*}
		g(C)&\le \frac{n}{m}g(F')+(\frac{n}{m}-1)(n-1)\\
		&\le \frac{n}{m}(m-1)^2+(\frac{n}{m}-1)(n-1)\\
		&=\frac{n^2}{m} -3n+nm+1\\
		&=:h_n(m).
	\end{align*}
	Because $m\mid n$, it suffices to prove that $h_n(m)\leq (n-1)^2$ for all $m\in [1,n]$.  For this, consider the function $h_n(t)=\frac{n^2}{t}-3n+nt+1$ as an analytic function of $t$ on the positive real line. For $t=1,n$, we have
	\[
		h_n(1)=h_n(n)=n^2-2n+1=(n-1)^2.
	\]
	Taking the derivative in $t$, we see that $h'_n(t)=n-\left(\frac{n}{t}\right)^2$, and thus $h_n(t)$ has a unique critical point in the positive reals at $t=\sqrt{n}\in [1,n]$.  Therefore, the maximum of $h_n(t)$ for $t\in [1,n]$ 
	occurs at $t=\sqrt{n}$ or at one of the endpoints $t=1,n$.  But for 
	$n>1$: 
	\begin{align*}
		h_n(\sqrt{n})&=2n\sqrt{n}-3n+1\\
		&=(2\sqrt{n}-1)n-2n+1\\
		&\le n^2-2n+1\\
		&=h_n(1)=h_n(n)=(n-1)^2. 
	\end{align*}
	We conclude that $h_n(t)\le (n-1)^2$ for all $t\in [1,n]$ and thus conclude the inductive step as claimed. 
\end{proof}

\begin{corollary}\label{c:key}
	Let $G$ be a finite group. Assume that $|G|\nmid n$ and that $G$ does not act nontrivially on an algebraic curve of genus at most $(n-1)^2$. Then no faithful irreducible $G$-curve $C$ admits a degree $n$ rational function. 
\end{corollary}
\begin{proof}
	Let $C$ be a faithful $G$ curve. Suppose the contrary, i.e. there exists a degree $n$ map $f\colon C\to \Pb^1$. For $g\in G$, let $f_g\colon C\to\Pb^1$ denote the map $x\mapsto f(gx)$.  Let 
	\[
		F=k(\{f_g\}_{g\in G})\subset k(C)
	\]
	denote the compositum. By construction, the field $F$ is $G$-invariant.  By Lemma~\ref{l:castn}, $g(F)\le (n-1)^2$.  Therefore, our assumption on $G$ implies that $F\subset k(C)^G$, i.e. $G$ fixes all the elements of $F$, and thus $F=k\left(\{f_g\}_{g\in G}\right)=k(f)$. But then,
	\begin{align*}
		n=[k(C):k(f)]&=[k(C):k(C)^G][k(C)^G:k(f)]\\
			&=|G|[k(C)^G:k(f)]
	\end{align*}
	which contradicts our assumption that $|G|\nmid n$.
\end{proof}

%The following is immediate from the definition of gonality and the fact that under a map of varieties, the image of an irreducible component of the source lies in a single irreducible component of the target.

We close this section with two additional lemmas.  First, a standard exercise with the field norm shows the following.
\begin{lemma}\label{l:lur}
	Let $C$ be an irreducible curve. Suppose there exists a dominant map $H\to C$ and a degree $n$ rational function $h\colon H\to \Pb^1$. Then $C$ has a degree $n$ rational function $f\colon C\to \Pb^1$.
\end{lemma}
\begin{proof}
	Let $h\colon H\to\Pb^1$ be a degree $n$ map. Let $N_{k(H)/k(C)}\colon k(H)^\times\to k(C)^\times$ denote the field norm.  Then $f:=N_{k(H)/k(C)}(h)\in k(C)$ is degree $n$. 
\end{proof}

%The following is standard, and recorded here for completeness.
%\begin{lemma}\label{l:res}
%	Let $f\colon E\to X$ be a degree $n$ branched cover. Then the restriction of $f$ to each irreducible component $E_i$ of $E$ is a degree $\le n$ map $f_{E_i}\colon E_i\to f(E_i)$.
%\end{lemma}
%\begin{proof}
 %   Degree is local in the target, so for any irreducible component $X_i\subset X_\Gamma$, $E|_{X_i}\to X_i$ is degree $n$. As the image of an irreducible component is irreducible, we see that $E|_{X_i}$ is a union of irreducible components of $E$. The restriction of $f$ to any one of them has degree at most $n$.
%\end{proof}

\section{Induced Actions on Unions of Rational Curves}
As above, we work over a perfect field $k$. Our main invariant for showing $\ed_k(-;\le n)>1$ comes from studying actions on unions of rational curves induced from a finite subgroup of $\PSL_2(k)$. We can now state and prove our key lemma.
\begin{lemma}\label{l:irred}
	Let $G$ be a finite group. Let $k$ be a perfect field. Suppose that:
    \begin{enumerate}
        \item $\ed_k(G;\le n)=1$,
        \item $G$ has no proper subgroup of index at most $n$, and
        \item $G$ contains a subgroup $M\subset G$ such that $M\into \PSL_2(k)$ and $|M|>n$.
    \end{enumerate}
    Then there exists a smooth, irreducible, projective faithful $G$-curve $\tilde{C}$ with a degree $m$ rational function $f\colon\tC\to\Pb^1$ for some $m\le n$.  
\end{lemma}
\begin{proof}  We work throughout in the birational category, i.e. the category of varieties and rational maps.  

Let $V$ be a faithful representation of $G$, viewed as a linear variety. 
By assumption, there exists a branched cover $\pi\colon E\dashrightarrow V/G,$ of degree $\leq n$ 
such that $\pi^*(V \rightarrow V/G)$ arises (rationally) by pullback from a $G$-cover of smooth projective curves  
$\tilde C \rightarrow C.$
Extend the inclusion of function fields $\kappa(V/G) \rightarrow \kappa(E)$ to an inclusion of separably closed fields 
$\bar \kappa(V/G) \rightarrow \bar \kappa(E),$ and consider the maps 
$$ \Gal(\bar \kappa(E)/ \kappa(E)) \rightarrow \Gal(\bar\kappa(V/G) \rightarrow \kappa(V/G)) \rightarrow G.$$ 
corresponding to $V \rightarrow V/G.$ The second map is surjective, and the image of the composite map 
has index $\leq n,$ as $\pi$ has degree $\leq n.$ By our assumption on $G$, we conclude that the composite map is surjective, and hence $\pi^*(V) \rightarrow E$ 
is an irreducible $G$-cover. In particular, $\tilde C$ is irreducible.

Next we remark that since $V \rightarrow V/G$ is $G$-versal, it is $M$-versal. 
Indeed, if $X$ is any $M$-variety, one may apply $G$-versality to $(X \times G)/M.$ 
Thus there is a rational, $M$-equivariant map $\Pb^1 \rightarrow V,$ corresponding to 
$Y: = \Pb^1/M \rightarrow V/G.$ Let $Y_E = \pi^*(Y),$ 
let $\kappa(Y_E) \supset \kappa(Y),$ be the function field at some generic point of $Y_E,$ 
and let $\bar \kappa(Y_E)$ be a separable closure of $\kappa(Y_E).$ Consider the composite 
$$ \Gal(\bar \kappa(Y_E)/ \kappa(Y_E)) \rightarrow \Gal(\bar \kappa(Y_E)/ \kappa(Y)) \rightarrow M.$$ 
The second map is surjective, and the image of the composite has index $\leq n.$ In particular, this image is 
non-trivial, as $|M| > n.$ Hence the composite map 
$$ Y_E \rightarrow E \rightarrow C $$
is non-constant, and so the corresponding $M$-equivariant map $\tilde Y_E : \pi^*(\Pb^1) \rightarrow \tilde C$ 
is nonconstant. As $\tilde Y_E \rightarrow \Pb^1$ has degree $\leq n,$ 
we conclude by Lemma~\ref{l:lur}, that $\tC$ admits a map $h\colon \tC\to\Pb^1$ of degree $\le n$.
	\end{proof}

\section{Finishing the proofs of Theorem~\ref{t:main} and Corollary~\ref{cor:main5}}\label{s:proofs}
We now complete the proofs of the theorems stated in the introduction.
\begin{proof}[Proof of Theorem~\ref{t:main}]
Let $n\ge 1$.  Let $G$ be a finite group satisfying the assumptions of the theorem.  Suppose to the contrary that $\ed_k(G;\le n)=1$. By Lemma~\ref{l:irred}, there exists a smooth, irreducible projective curve $\tC$ with a faithful $G$-action and a degree $m$ rational function for some $m\le n$. By Corollary~\ref{c:key}, this implies that $G$ acts nontrivially on a curve of genus at most $(m-1)^2\le (n-1)^2$. But this contradicts Assumption~(\ref{a:cast}) of the theorem.  Thus $\ed_k(G;\le n)>1$.
\end{proof}

There are infinitely many examples of $(k,G,n)$ to which Theorem~\ref{t:main} applies: e.g. for $k=\Cb$, $G$ simple, and for 
\begin{equation}\label{e:simple}
    n\le \min\{d(G),\max\{m~|~C_{m+1}\subset G,1+\sqrt{1+|G|/84}\}\}
\end{equation} where $d(G)$ denotes the size of the smallest permutation representation of $G$.  For the finite groups $G$ of classical type, a complete list of $d(G)$ is given in \cite[Table 1]{Cooperstein}. For every finite simple group $G$ the number 
$d(G)$ can be extracted from the classification of finite simple groups (e.g. see the Atlas \cite{ATLAS} for the sporadic simple groups), and by Cauchy's Lemma, one can replace the max over cyclic subgroups in~\eqref{e:simple} by $p-1$ for $p$ the largest prime dividing $|G|$. As the labeling implies, Corollary~\ref{cor:main5} gives a set of such examples. 
\begin{proof}[Proof of Corollary~\ref{cor:main5}]
    We start with statement~\ref{i:mKl}. Let $G$ be a non-abelian finite simple group not isomorphic to $A_5$. As every index 2 subgroup is normal, $G$ has no proper subgroup of index 2.  Further, there exists an odd prime $p$ such that $p\mid |G|$, and thus by Cauchy's lemma, a cyclic subgroup $C_p\subset G$ of order greater than 2. By Theorem~\ref{t:main}, it suffices to prove that $G$ does not act on an elliptic or rational curve over $\Cb$. Because the hyperelliptic involution is unique and central in the automorphism group of a hyperelliptic curve \cite[Ch. III, Corollaries 2, 3, p. 108]{FaKr}, we see that every group acting faithfully on a rational, elliptic or hyperelliptic curve is a subquotient of a $C_2$-central extension
    \[
	   1\to C_2\to G\to \bar{G}\to 1
    \]
    with $\bar{G}\subset \PGL_2(\Cb)$. Combining Theorem \ref{t:main} with Klein's \cite{KleinIcos} classification of finite M\"obius groups, we obtain statement~\ref{i:mKl}.

For statement~\ref{i:mA7}, note that $A_7$ has no subgroup of index less than 7, and there exists a $C_7\subset A_7$ (pick a 7-cycle). The argument above combines with the Hurwitz bound to show that $A_7$ does not act nontrivially on any curve of genus less than 31.  But $31>(6-1)^2=25$, so the statement follows from Theorem~\ref{t:main}.

    For the final statement~\ref{i:mPSL2}, as observed above, for $p\ge 7$ the group $\PSL_2(\F_p)$ is simple and does not act on a rational or elliptic curve. Therefore, by the Hurwitz bound, $\PSL_2(\F_p)$ does not act on a curve of genus less than $|\PSL_2(\F_p)|/84+1$. It remains to verify the first two assumptions of Theorem~\ref{t:main}. For $p=7,11$, $\PSL_2(\F_p)$ has no subgroup of index less than $p$, and the upper bound on $n$ above is equal to $1+\lfloor\sqrt{1+\frac{p(p^2-1)}{168}}<p-1$. For $p>11$, Galois showed that $\PSL_2(\F_p)$ has no subgroup of index less than $p+1$ (cf. \cite[p. 213]{Cooperstein}). In all cases, we see that $\PSL_2(\F_p)$ satisfies the first assumption in the statement of Theorem~\ref{t:main}.  For the second assumption, note that for all $p>2$, we have $|\PSL_2(\F_p)|=\frac{p(p^2-1)}{2}$.  By Cauchy's Lemma, $C_p\subset \PSL_2(\F_p)$. In each case, we conclude by Theorem~\ref{t:main} that $\ed_{\Cb}(\PSL_2(\F_p);\le n)>1$.
\end{proof}

\end{document}